\documentclass[12pt]{article}

\usepackage{epsfig,amsmath,amsfonts,amssymb,latexsym,color,amsthm,hhline,multirow,subfigure,graphicx}

\newcommand\marginal[1]{\marginpar{\raggedright\parindent=0pt\tiny #1}}

\newtheorem{theorem}{Theorem}[section]
\newtheorem{prop}[theorem]{Proposition}
\newtheorem{lemma}[theorem]{Lemma}

\newtheorem{claim}[theorem]{Claim}

\theoremstyle{definition}
\newtheorem*{remark}{Remark}

\newcommand{\E}{{\mathbb E}}

\newcommand {\PP}{{\mathbb P}}
\newcommand{\sss}{\scriptscriptstyle}
\newcommand{\cSi}{\mathcal{S}^{\sss(i)}}
\newcommand{\cBi}{\mathcal{B}^{\sss(i)}}
\newcommand{\cBe}{\mathcal{B}^{\sss(1)}}
\newcommand{\cBt}{\mathcal{B}^{\sss(2)}}
\newcommand{\cSe}{\mathcal{S}^{\sss(1)}}
\newcommand{\cSt}{\mathcal{S}^{\sss(2)}}
\newcommand{\cBhi}{\widehat{\mathcal{B}}^{\sss(i)}}
\newcommand{\cBhe}{\widehat{\mathcal{B}}^{\sss(1)}}
\newcommand{\cBht}{\widehat{\mathcal{B}}^{\sss(2)}}
\newcommand{\Si}{S^{\sss (i)}}
\newcommand{\Se}{S^{\sss (1)}}

\newcommand{\cNh}{\widehat{\mathcal{N}}}

\numberwithin{equation}{section}

\newcommand{\sumk}{\sum_{k=1}^\infty}

\newcommand{\sumin}{\sum_{i=1}^n}

\newcommand\set[1]{\ensuremath{\{#1\}}}
\newcommand\bigset[1]{\ensuremath{\bigl\{#1\bigr\}}}

\newcommand\bigpar[1]{\bigl(#1\bigr)}
\newcommand\Bigpar[1]{\Bigl(#1\Bigr)}

\newcommand\bigabs[1]{\bigl|#1\bigr|}

\def\rompar(#1){\textup(#1\textup)}    % usage: \rompar(...)

\def\xexp(#1){e^{#1}}

\newcommand\floor[1]{\lfloor#1\rfloor}

\newcommand\ntoo{\ensuremath{{n\to\infty}}}

\newcommand\ktoo{\ensuremath{{k\to\infty}}}

\newcommand\ttoo{\ensuremath{{t\to\infty}}}

\newcommand\punkt{.\spacefactor=1000}    % om problem!

\newcommand\whp{w.h.p\punkt}
\newcommand\whpx{w.h.p}

\newcommand{\tend}{\longrightarrow}
\newcommand\dto{\overset{\mathrm{d}}{\tend}}
\newcommand\pto{\overset{\mathrm{p}}{\tend}}
\newcommand\asto{\overset{\mathrm{a.s.}}{\tend}}

\newcommand\bbN{\mathbb N}

\renewcommand\P{\operatorname{\mathbb P{}}}

\newcommand\gD{\Delta}

\newcommand\gl{\lambda}

\newcommand\gs{\sigma}

\newcommand\eps{\varepsilon}

\renewcommand\phi{\xxx}  %% WARNING

\newcommand\cF{\mathcal F}

\newcommand\cM{\mathcal M}
\newcommand\cN{\mathcal N}

\newcommand\cP{\mathcal P}

\newcommand\cS{{\mathcal S}}

\newcommand\cV{\mathcal V}

\newcommand\tM{\tilde M}

\newcommand\qw{^{-1}}

\newcommand\qqq{^{1/3}}

\newcommand{\refL}[1]{Lemma~\ref{#1}}

\newcommand\REM[1]{{\raggedright\texttt{[#1]}\par\marginal{XXX}}}

\newcommand\cFx{\cF^+}

\newcommand\kk{^{(k)}}

\newcommand\bigmid{\,\big|\,}
\begin{document}

\title{Competing first passage percolation on random graphs with finite variance degrees\thanks{The authors are grateful to Olle Nerman for pointing out some relevant references. This work was in part supported by the Swedish Research Council (grants 237-2013-7302 and 2016-04442 DA, MD) and by the Knut and Alice Wallenberg Foundation (SJ).}}\parskip=5pt plus1pt minus1pt \parindent=0pt

\author{Daniel Ahlberg\thanks{Department of Mathematics, Stockholm University; {\tt daniel.ahlberg@math.su.se}}  \and Maria Deijfen\thanks{Department of Mathematics, Stockholm University; {\tt mia@math.su.se}} \and Svante Janson\thanks{Department of Mathematics, Uppsala University; {\tt svante.janson@math.uu.se}}  }

\date{8 November 2017}

\maketitle

\begin{abstract}
\noindent We study the growth of two competing infection types on graphs generated by the configuration model with a given degree sequence. Starting from two vertices chosen uniformly at random, the infection types spread via the edges in the graph in that an uninfected vertex becomes type 1 (2) infected at rate $\lambda_1$ ($\lambda_2$) times the number of nearest neighbors of type 1 (2).
Assuming (essentially) that the degree of a randomly chosen vertex has finite second moment, we show that if $\lambda_1=\lambda_2$, then the fraction of
vertices that are ultimately infected by type 1 converges to a continuous random variable $V\in(0,1)$, as the number of vertices tends to
infinity. Both infection types hence occupy a positive (random) fraction of the vertices. If $\lambda_1\neq \lambda_2$, on the other hand, then the type with the larger intensity
occupies all but a vanishing fraction of the vertices. Our results apply also to a uniformly chosen simple graph with the given degree sequence. 

\noindent
\vspace{0.3cm}


\vspace{0.2cm}


\end{abstract}

\section{Introduction}

Fix $n\geq 1$ and let $(d_1,\ldots,d_n)$ be a sequence of positive integers
that may depend on $n$. Consider a graph with $n$ vertices and degrees
$(d_1,\ldots,d_n)$ generated by the configuration model, that is, equip each
vertex $i\in\{1,\ldots,n\}$ with $d_i$ half-edges, and pair half-edges
uniformly at random to create edges. For all half-edges to find a partner we
assume that the total degree $\sum d_i$ is even. Assign independently to each edge $e$ in the resulting graph two independent
exponentially distributed passage times $X_1(e)$ and $X_2(e)$ with parameter
$\lambda_1$ and $\lambda_2$, respectively. At time 0, two uniformly chosen
vertices are infected with infections type 1 and type 2, respectively, and
the infections then spread via nearest neighbors: When a vertex becomes type
1 (2) infected, the time that it takes for the infection to traverse an edge
$e$ emanating from the vertex is given by $X_1(e)$ ($X_2(e)$). If the other
end point of the edge $e$ is still uninfected at that time, it becomes type
1 (2) infected and remains so forever. It also becomes immune to the other
infection type.

In this paper we study the above competing growth process on a random graph
generated from a given degree sequence subject to the regularity conditions
stated below. These conditions ensure that the graph contains a giant component occupying all but a vanishing fraction of the vertices as $n\to\infty$, and hence that almost all vertices will \whp{} be infected when the process terminates. The question that we will be interested in is the outcome of this competition. Specifically, will both types occupy a strictly positive fraction of the vertices in the limit as
$n\to\infty$? We show that the answer is yes if and only if
$\lambda_1=\lambda_2$. This question has previously been studied for the
configuration model with constant degrees~\cite{regular} and infinite
variance degrees~\cite{winner}; see the end of this section for a summary of
earlier work.

Given a degree sequence $(d_1^{\sss(n)},d_2^{\sss(n)},\ldots,d_n^{\sss(n)})$ with $\sum d_i^{\sss(n)}$ even, write $D_n$ for the degree of a vertex chosen uniformly at random, so that
$$
\P(D_n=k)=\#\{i:d_i=k\}/n.
$$
Our assumptions on the (sequence of) degree sequences are the following: 
\begin{itemize}
\item[(A1)] $(D_n)_{n\ge1}$ converges in distribution to a random variable $D$ with $\E[D^2]<\infty$, and
$$
\E[D_n^2]\to\E[D^2];
$$
\item[(A2)]  $d_i\geq 2$ for all $i$, and $\P(D>2)>0$.
\end{itemize}

Assumption (A1) could equivalently be formulated as the sequence of empirical distributions being uniformly square integrable and converging to a probability distribution $(p_d)_{d\in\mathbb{N}}$ on the positive integers. One standard example in which (A1) is satisfied is when $(d_1, d_2,\ldots, d_n)$ are independent realizations of a random variable $D$ with finite variance. By increasing a randomly chosen degree by 1, if necessary, we can make sure that the total degree is even. If we condition on the sequence $(D_i)_{i=1}^n$ and assume that $\E[D^2]<\infty$, $\P(D\ge2)=1$ and $\P(D>2)>0$, then (A1) and (A2) hold \whp{} and thus our results, as stated below, apply.

A graph generated by the configuration model may contain self-loops and multiple edges, but the assumption (A1) implies that the probability of obtaining a simple graph is bounded away from 0 as $n\to\infty$; see \cite{AngHofHol16,simpleI,simpleII}. Furthermore, it is well-known that conditioning on the resulting graph being simple yields a uniform sample among simple graphs with the specified degree sequence; see \cite[Chapter 7]{Remco_book}. Hence our results apply also for such a uniformly chosen simple graph.

Let $D^*$ be a size biased version of $D$, that is,
$\P(D^*=d)=d\P(D=d)/\E[D]$. The threshold for the occurrence of a (unique)
giant component in the graph is given by $\E[D^*-1]=1$; see
\cite{SvanteMalwina,MR-95}. This can be seen by exploring the components in
the graph via nearest neighbors, starting from a uniformly chosen vertex. As
$n\to\infty$, this exploration can be approximated by a branching process
and, by construction of the graph, the offspring distribution of explored
vertices in the second and later generations is given by $D^*-1$. The
relative size of the giant component is given by the survival probability in
the approximating branching process; see \cite{SvanteMalwina,MR-98}. Condition
(A2) above implies that the survival probability is 1, so that the
asymptotic fraction of vertices in the giant component is 1.

Now consider the competition process described above. Write $N_i(n)$ for the total number of type $i$ infected vertices when the process terminates, and $\bar{N}_i(n)=N_i(n)/n$ for the corresponding fraction. Note that, since the giant component spans all but a vanishing fraction of the vertices, we have that $\bar{N}_1(n)+\bar{N}_2(n)\pto 1$, and it is therefore enough to consider $\bar{N}_1(n)$. Furthermore, by symmetry, we may assume that $\lambda_1\leq \lambda_2$. The following is our main result.

\begin{theorem}\label{th:main} Assume that the degree sequence satisfies (A1) and (A2).
\begin{itemize}
\item[{\rm{(a)}}] If $\lambda_1=\lambda_2$, then $\bar{N}_1(n)\dto V$, where $V$ is a continuous random variable with a strictly positive density on $(0,1)$.
\item[{\rm{(b)}}] If $\lambda_1<\lambda_2$, then $\bar{N}_1(n)\dto 0$.
\end{itemize}
\end{theorem}

\begin{remark}
Starting with two given infected vertices, e.g.\ vertices 1 and 2, or
several infected vertices of each type (fixed in number as $n\to\infty$)
gives the same results, except that the distribution of the limiting
fraction $V$ will depend on the degrees of the initially infected
vertices. Moreover, the theorem extends to a fixed number of competing types
larger than two, in which case all types of maximal strength each conquer a
positive fraction of the vertices.
\end{remark}

\begin{remark}
The assumption $d_i\geq 2$ ensures that the giant component
comprises almost all vertices. Weakening this condition to $\E[D^*-1]>1$
gives a graph where the giant component may contain a smaller fraction of
the vertices. The competition process can be analyzed also on such a graph
and the non-trivial case then arises when both initial vertices belong to
the giant component. We believe that our methods apply also in this case,
but it would require dealing with a conditioning on both initial vertices
being in the giant component. Establishing a  version of Theorem
\ref{th:main} in that case would make it applicable also for e.g.\ the
Erd\H{o}s--Renyi graph and the generalized random graph analyzed in
\cite{BDM-L}. These models give simple graphs with random degrees and,
conditionally on the degrees, the graph is uniform on the set of all simple
graphs with those given degrees.
\end{remark}

\textbf{Outline of the proof}

In the proof below we establish that there is an initial phase where the
outcome of the competition is determined, followed by a phase that lasts until
close to the end, and where the fractions of two types are essentially
constant. An important tool in the proof is a standard technique
for exploring the graph and the evolution of the infections
simultaneously. A vertex is detected when it is reached by the infection and
the half-edges attached to the vertex are then declared active, of type 1 or
2 depending on the type of the vertex. A half-edge remains active until it
is opened for infection. A partner half-edge is then chosen and, if the
vertex of this half-edge is still uninfected at that time, this leads to
infection transfer and activation of new half-edges. The process can be
defined in continuous time or in discrete steps by observing it only at the
time points when an edge is opened; see Section 2 for a more detailed
description. Write $S^{\sss (i)}_k$ for the number of active type $i$
half-edges after $k$ steps in this process and $S_k=S^{\sss (1)}_k+S^{\sss (2)}_k$ for the total number active half-edges. Define $M_k$ to be the fraction of active type 1 half-edges among all active half-edges, 
more precisely defined by
\begin{equation}\label{eq:Mk}
M_k:=\begin{cases}
\frac{S^{\sss (1)}_k}{S_k} & \text{if } S_k>0;\\
M_{k-1} & \text{if }S_k=0.
  \end{cases}
\end{equation}

In a key step we show that, if $\lambda_1=\lambda_2$, then $M_k$ is a martingale. We then give an estimate of its quadratic variation which implies that $M_k$ is essentially constant for $k\geq \nu_n$ for any sequence of integers $\nu_n\to\infty$. The probability that a newly infected vertex is infected by type 1 is hence roughly constant for $k\geq \nu_n$ and equal to $M_{\nu_n}$. The initial stages of the
competition, on the other hand, can be approximated by a branching process
and asymptotic results on branching processes imply that $M_{\nu_n}$
converges to a continuous random variable $V\in(0,1)$ if $\lambda_1=\lambda_2$, and to 0 if $\lambda_1<\lambda_2$. This yields Theorem \ref{th:main}(a). The proof of Theorem \ref{th:main}(b) is completed
by letting the weaker type 1 infection spread with the same larger intensity
$\lambda_2$ as the type 2 infection for $k\geq\nu_n$. The fraction of type 1
vertices among infected vertices for $k\geq \nu_n$ in such a process is
close to 0 by the above results, and the type 1 infection clearly captures
even fewer vertices in the original process.

The rest of the paper is organized so that the exploration process is described in more detail in Section 2, along with the initial branching process approximation. The results on $M_k$, specifying the evolution of the infections during the main phase, are then given in Section 3. Theorem \ref{th:main} is proved in Section 4. Finally some directions for future work are described in Section 5.\medskip

\textbf{Previous work}

Competition on the configuration model has previously been studied in the
case when the degree distribution follows a power-law with exponent
$\tau\in(2,3)$, that is, when the mean degree is finite, but the variance
infinite. In that case one of the types occupies all but a finite number of
vertices as $n\to\infty$, and both types have a positive probability of
winning, regardless of the values of the intensities; see \cite{winner}. The process has also been studied on random regular graphs generated by the
configuration model with constant degree; see \cite{regular}. Our results
generalize the results in \cite{regular} when the competition starts from
fixed initial sets. However, the results in \cite{regular} also cover the
case with growing initial sets, and give precise quantifications of the
asymptotic number of vertices of each type. 

In the present work, as well as in \cite{regular,winner}, the passage times
are assumed to be exponential. The model can of course be defined
analogously for passage times with arbitrary distributions. It has been
analyzed in \cite{fixspeedI,fixspeedII} for configuration graphs with
power-law exponent $\tau\in(2,3)$ and constant passage times, so that all
randomness comes from the underlying graph. When the types have different
speed, the faster type occupies all but a vanishing fraction of the
vertices, while when the speeds are the same, the types may or may not
occupy positive fractions depending on the specific choice of the two
initial vertices. A slightly different competition process with constant
passage times is analyzed in \cite{Cooper}, and the present competition
process is analyzed on preferential attachment graphs in \cite{prefatt}. 

Finally, we mention that competing first passage percolation with
exponential passage times has previously been studied on $\mathbb{Z}^d$. In
that setting coexistence may occur for equal strength competitors, whereas
the case of unequal strength remains to be fully resolved; see \cite{2tRich}
for a survey and references.

\section{The initial phase}

In this section we first define the exploration of the graph and the flow of infection in more detail. We then describe a branching process approximation of the number of active half-edges of the two types in the early stages of the growth. This leads to a characterization of the limiting behavior of a continuous time version of $M_k$ (defined in \eqref{eq:Mk}) at the end of the initial phase. \medskip

\textbf{The exploration process}

To describe the exploration process, fix $\lambda_1,\lambda_2>0$, possibly different. At time 0 we start with the vertices and the attached half-edges. The pairing of the half-edges however is hidden and is revealed during the process. Each half-edge is throughout the process classified as either \emph{free} or \emph{paired}, and a free half-edge is in turn labeled as \emph{active} of either type 1 or 2, or \emph{inactive}. The initial set of active type $i$ half-edges consists of the half-edges attached to the uniformly chosen initial type $i$ vertex, while all other half-edges are inactive. Since the initially infected vertices are chosen randomly, the initial numbers $a_1$ and $a_2$ of active type 1 and type 2 half-edges, respectively, are random. However, we condition on them in the sequel, and hence assume that they are given numbers. The sets of half-edges are now updated inductively in continuous time as follows, with $\cSi_t$ denoting the number of active half-edges of type $i$ at time $t$.

\begin{enumerate}
\addtolength{\leftmargini}{-5pt}%
 \renewcommand{\labelenumi}{\textup{(P\arabic{enumi})}}%
 \renewcommand{\theenumi}{\labelenumi}%

\item \label{sj1}
Each active half-edge of type $i=1,2$ infects with intensity $\gl_i$, that is, it is equipped with an exponential clock with intensity $\gl_i$, and infects when the clock rings. When a half-edge $q$ infects, it picks a partner $r$  uniformly at random from all free half-edges distinct from $q$. Let $x$ and $y$ be the vertices that $q$ and $r$, respectively, are attached to. Then $q$ and $r$ go from free to paired and form an edge $xy$.

\item \label{sjold}
If $y$ is already infected, nothing more happens. In this case, $r$ was also active (of the same type as $q$ or not), and the number of active half-edges decreases by 2.

\item \label{sjnew}
If $y$ is not infected, it becomes infected by the same type as $x$, and all
remaining half-edges at $y$ become active of this type. This means that, if $q$ has type $i$ and $y$ has degree $d_y$, then the number $\cSi_t$ of active half-edges of type $i$ increases by $d_y-2$, while the number of active half-edges of the other type does not change.
\end{enumerate}

A discrete version of the process is obtained by observing the continuous time process at the times half-edges are paired. In each step $k$ of the discrete time process, an active half-edge $q$ is
chosen at random, with probability proportional to $\gl_i$ where $i$ is its
type. The chosen half-edge infects as in \ref{sj1}--\ref{sjnew} above. In
both cases, if there are no remaining active half-edges, the infections have
stopped, but it still remains to complete the graph. We then join any two uniformly chosen half-edges, that is, we choose a uniform matching of the remaining half-edges. The number of active type $i$ half-edges after $k$ steps is denoted by $\Si_k$. Throughout, quantities related to discrete time processes will be denoted by standard roman letter, while quantities related to processes in
continuous time will be denoted by calligraphy letter. For instance $\cM_t$
denotes the continuous time version of $M_k$, defined in \eqref{eq:Mk}, that
is, $\cM_t=\cSe_t/(\cSe_t+\cSt_t)$.\medskip

\textbf{Branching process approximation} 

We now describe how the early evolution of $\cSi_t$ ($i=1,2$) can be coupled with two independent branching processes. Stronger results in this direction have been obtained in \cite{one_fpp}. However, we only need the coupling up to some time $t_n\to\infty$ (without further restrictions on $t_n$). This is fairly easy to establish and we therefore describe it here. 

Our aim is to prove the following result on the fraction of active type 1 half-edges in the initial phase.

\begin{prop}\label{prop:initial} There exists a deterministic sequence of integers $t_n\to \infty$ such that $\cM_{t_n}\dto V$ as $n\to\infty$, where $V$ is a continuous random variable with strictly positive density on $(0,1)$ if $\lambda_1=\lambda_2$ and $V\equiv 0$ if $\lambda_1<\lambda_2$.
\end{prop}

We now consider the initial phase of the continuous time exploration process when $t$ is so small that rather few vertices have been infected. First consider the general case with $\gl_1,\gl_2>0$, possibly different, and the process described by \ref{sj1}--\ref{sjnew} above. In order to study the initial phase, we introduce the corresponding process where half-edges in \ref{sj1} are drawn with replacement, that is, the half-edge $r$ is chosen uniformly at random from the set of \emph{all} half-edges, independently of previous picks. In this version we do not have to keep track of the actual sets of active half-edges, only their numbers, which we denote by $\cBe_t$ and $\cBt_t$. Moreover, we pretend that the chosen half-edge and its vertex are not used before, so we ignore \ref{sjold} and always update $\cBe$ and $\cBt$ as in \ref{sjnew}. This means that $\cBe$ and $\cBt$ are two independent continuous time Markov branching processes with intensities $\gl_1$ and $\gl_2$, respectively, and the same offspring distribution $D_n^*-1$, where $D_n^*$ is the size-biased distribution corresponding to the empirical distribution $D_n$, that is, $\P(D_n^*=d):=d\P(D_n=d)/\E[D_n]$. Of course, we take $\cBi_0=a_i$.

Furthermore, define $\cBhi_t$ to be a branching process defined as $\cBi_t$ but with
the offspring distribution changed to $D^*-1$. Thus $\cBhi_t$, unlike $\cSi_t$ and $\cBi_t$, does not depend on $n$. Since $\E [D^*-1]=\E[D(D-1)]/\E[D]<\infty$, there is no explosion, and $\cBhi_t$ is a.s.\ finite for all $t$. Specifically, for every fixed $T<\infty$, the process $\cBhi$ has a.s.\ only a finite number of births (infections) in $[0,T]$. Moreover, since $D_n\dto D$ and $\E[D_n]\to\E [D]<\infty$, we have that $D_n^*\dto D^*$. It follows that, for every fixed $T<\infty$, we can couple $\cBi$ and $\cBhi$ such that they agree with probability $1-o(1)$ each time an individual gets offspring at a time $t\leq T$, that is, \whp{} $\cBi_t=\cBhi_t$ for all $t\le T$.

Now return to the actual exploration process. We can obtain it from the
version with replacement by accepting a selected half-edge $r$ if it is
free, and otherwise resampling. Moreover, we also check if the accepted
half-edge already is active, and then we apply \ref{sjold} instead of
\ref{sjnew}. During a fixed time interval $[0,T]$, the process $\cBhi_t$
has a.s.\ only finitely many births and thus, since $\cBi_t=\cBhi_t$ \whp{} on this
interval, the number of births in $[0,T]$ for $\cBi_t$ is $O_p(1)$. Furthermore, the number of half-edges that are paired in $[0,T]$ is $O_p(1)$, and so is the number of half-edges that are declared active in $[0,T]$. Hence, at each of the $O_p(1)$  births in $[0,T]$, the probability that a paired or active half-edge is picked in the process $\cBi_t$ is $o(1)$. Consequently, \whp, only free inactive half-edges are selected in $\cBi_t$ for $t\le T$ and the process then agrees completely with $\cSi_t$ for $t\le T$.

We have shown that the processes $\cSi_t$ and $\cBhi_t$ can be coupled (for $i=1,2$ simultaneously) such that, for every fixed $T$, we have that $\cSi_t=\cBhi_t$ for $t\le T$. Let
$$
\tau_n:=\inf\bigset{t\ge0:\cSi_t\neq\cBhi_t \text{ for some $i\in\set{1,2}$}}.
$$
It follows that $\P(\tau_n\le T)\to0$ for every fixed $T$, that is, $\tau_n\pto\infty$. This implies that there is a deterministic sequence $t_n\to\infty$ such that $\P(\tau_n\le t_n)\to0$. In other words,
\whp{}
\begin{equation}\label{x=z}
\cSe_t=\cBhe_t
\quad\text{and}\quad
\cSt=\cBht_t\quad
\text{for } t\le t_n.
\end{equation}
Fix such a sequence $t_n\to\infty$ where, for later use, we pick the
sequence such that each $t_n$ is an integer. 
For the proof of Theorem
\ref{th:main}, it will be useful to adjust the sequence slightly to ensure
that the number of vertices that have been infected at time $t_n$ is small.
Thus, let $\cN_t$ be the number of edges
identified in the exploration process at time $t$; this equals the number of
times that \ref{sj1} has been performed. Also let $\cNh_t$ be the analogous
quantity for the process $\cBhe_t\cup\cBht_t$. With the coupling above, we
have  $\cN_t=\cNh_t$ for $t<\tau_n$, and hence \whp{}
$\cN_{t_n}=\cNh_{t_n}$. We may assume, by decreasing $t_n$ if necessary,
that $\cNh_{t_n}\le n\qqq$ \whpx.

We also define a related sequence of integers $\nu_n$ such that, 
in the discrete time exploration process, the branching process approximation remains valid beyond step $\nu_n$.
To do this, note that $\cNh_{t_n}\asto\infty$ as
\ntoo, since $t_n\to\infty$. Hence, $\cNh_{t_n}\pto\infty$ and
$\cN_{t_n}\pto\infty$, and thus there exists a deterministic sequence
$\nu_n$ of integers such that $\nu_n\to\infty$ and \whp{}
\begin{equation}\label{nun}
n\qqq\ge\cNh_{t_n}= \cN_{t_n}\ge \nu_n.
\end{equation}

Finally note that, by our assumptions, $D_n^*\ge2$ and thus $D^*\ge 2$ so that $D^*-1\ge1$. This means that the branching processes $\cBhi_t$ never decrease. In particular, they never become extinct, and therefore $\cBhi_t\to\infty$ a.s.\ as \ttoo.

With the above coupling at hand we can prove Proposition \ref{prop:initial}.

\begin{proof}[Proof of Proposition \ref{prop:initial}]
Suppose first that $\gl_1=\gl_2$. The branching processes $\cBhe_t$ and $\cBht_t$ are independent and have the same offspring distribution, but possibly different initial values $a_1$ and $a_2$. If we restrict to integer values of $t$, we obtain two independent Galton--Watson processes $\cBhe_k$ and $\cBht_k$ with the same offspring distribution. Moreover, this offspring distribution has a finite mean $m>1$, since, by assumption, $\E[D^2]<\infty$ and thus $\E[D^*]<\infty$ (in fact we have $m=e^{\E[D^*-1]}$). By the Seneta--Heyde theorem \cite{Heyde} (see also \cite[Theorem I.10.3]{AthreyaNey}) there exists a derministic sequence $c_k$ such that $\cBhi_k/c_k\to W_i$ a.s., where $W_i\in(0,\infty)$ is a random variable, and thus
$$
\frac{\cBhe_k}{\cBhe_k+\cBht_k} \asto V
$$
for some random variable $V\in(0,1)$. By \cite[Theorem II.5.2]{AthreyaNey} and the subsequent remark, the variable $W_i$ ($i=1,2$) is continuous with strictly positive density on $(0,\infty)$ and hence $V$ is continuous with strictly positive density on $(0,1)$. Since $t_n\to\infty$, and we have assumed that $t_n\in\bbN$, it follows that
\begin{equation}\label{Szlim}
\frac{\cBhe_{t_n}}{\cBhe_{t_n}+\cBht_{t_n}} \asto V \in(0,1)
\end{equation}
as \ntoo. Alternatively, we can use the continuous-time version of the Seneta--Heyde theorem by Cohn \cite{Cohn} to directly arrive at~\eqref{Szlim}. Since $\cSi_{t_n}=\cBhi_{t_n}$ \whp\, by \eqref{x=z}, it follows from \eqref{Szlim} that
\begin{equation}\label{Sxlim}
\cM_{t_n}=  \frac{\cSe_{t_n}}{\cSe_{t_n}+\cSt_{t_n}} \dto V\in(0,1),
\end{equation}
and the first part of Proposition \ref{prop:initial} is proved.

Now suppose that $\gl_1<\gl_2$. By time-scaling we may assume that $\gl_1=1$ and $\gl_2=\gl>1$. Then $\cBhe_{\gl t}$ and $\cBht_t$ are two independent continuous time branching processes, with the same intensity and the same offspring distribution (with finite mean). Hence, as in the case with equal intensities, there exist $c_k$ such that a.s.\
\begin{align}
\cBhe_{\gl k}/c_k&\to W_1 \label{ax1}
\\
\cBht_k/c_k&\to W_2,  \label{ax2}
\end{align}
where $W_1$ and $W_2$ are random variables with $W_i\in(0,\infty)$
a.s. Furthermore, $c_{k+1}/c_k\to m>1$. For any fixed $j\ge0$, we have for large enough $k$ that $\gl k\ge k+j$, and thus $\cBhe_{k+j}\le\cBhe_{\gl k}$. Hence, by \eqref{ax1}, a.s.
$$
\limsup_\ktoo\frac{\cBhe_k}{c_k}
=
\limsup_\ktoo\frac{\cBhe_{k+j}}{c_{k+j}}
\le
\limsup_\ntoo\frac{\cBhe_{\gl k}}{c_k}\cdot\frac{c_k}{c_{k+j}}
=W_1 m^{-j}.
$$
Since $W_1<\infty$, $m>1$ and $j\ge0$ is arbitrary, it follows that $\limsup_\ktoo\cBhe_k/c_k=0$ a.s.\ and thus, recalling from \eqref{ax2} that $\cBht_k/c_k\to W_2>0$, that $\cBhe_k/\cBht_k\asto 0$. Hence, \eqref{Szlim} and \eqref{Sxlim} hold with $V\equiv 0$.
\end{proof}

\section{The deterministic phase}

In this section we show that the fraction $M_k$ of active type 1 half-edges among all active half-edges remains roughly constant after the initial phase in the exploration process for equal intensities. At the very end of the process, when most half-edges have already been paired, this might fail, but we show that the fraction is indeed constant during the main part of the process. Here we will work mainly in discrete time, and then connect to continuous time in the proof of Theorem \ref{th:main}. We denote the total number of edges in the graph by $N$, that is,
$$
N=\frac{1}{2}\sum_id_i;
$$
this is the total number of steps in the discrete time exploration process.

\begin{prop}\label{prop:main} Assume that $\lambda_1=\lambda_2=1$ and let $\nu_n$ be defined as in \eqref{nun}. As \ntoo\, we have for any $\eps>0$ that
  \begin{equation}\label{lx1}
    \sup_{\nu_n\le k\le (1-\eps)N}\bigabs{M_k-M_{\nu_n}} \pto 0.
  \end{equation}
\end{prop}

\begin{remark}
Proposition \ref{prop:main} is valid for any sequence $\nu_n\to\infty$ with $\nu_n\le(1-\eps)N$. However, we will apply it to the sequence $\nu_n$ defined in \eqref{nun} and therefore formulate it for
this. The idea is that the branching process approximation in Section 2
remains valid beyond step $\nu_n$ in the discrete process, and Proposition
\ref{prop:main} then ensures that the proportion of type 1 vertices does not
change after that.
\end{remark}

The key observation in the proof of Proposition \ref{prop:main} is that $M_k$ is a martingale when $\lambda_1=\lambda_2$. We then show that the second moment assumption implies that the contribution to the quadratic variation of this martingale during the range $\nu_n$ to $(1-\eps)N$ is vanishingly small. With this at hand it is not hard to show \eqref{lx1}.

\begin{lemma}\label{LM}
If $\lambda_1=\lambda_2$, then $(M_k)_{k=0}^N$ is a martingale.
\end{lemma}

\begin{proof} Recall that $S_k$ denotes the total number of active half-edges after $k$ steps. Define $\gD S_k=S_{k+1}-S_k$, and similarly for other sequences.

Let $\cF_k$ be the $\gs$-field generated by all events up to step $k$. Next, reveal whether a new vertex is infected in step $k$, and if so, the identity (and thus the degree) of the new infected vertex (however, we do not yet reveal the classification of the involved half-edges). Let $\cF_k^+\supset \cF_k$ denote the $\gs$-field generated by the events revealed so far. If a new node of degree $d$ is infected, then $\gD S_k=d-2$, and $\gD \Se_k$ is either $d-2$ or 0, with conditional probabilities (given $\cF_k^+$) $M_k$ and $1-M_k$, respectively. Hence, in this case,
$$
\E\left[\gD \Se_k\mid\cF_k^+\right] = M_k(d-2)
$$
and thus
$$
\E\left[\Se_{k+1}\mid\cF_k^+\right] = \Se_k+M_k(d-2)
= M_k(S_k+d-2)=M_kS_{k+1};
$$
Hence, $\E\bigpar{M_{k+1}\mid\cFx_k}=M_k$. If no new vertex is infected, and
$S_k>0$, then $\gD S_k=-2$. Since the two paired half-edges are then both drawn uniformly at random (without replacement) from the active half-edges,
each one of them has (conditional) probability $M_k$ of being of type 1. Hence
$$
\E\left[\gD \Se_k\mid\cF_k^+\right] = -2M_k
$$
and thus
$$
\E\left[\Se_{k+1}\mid\cF_k^+\right] = \Se_k-2M_k= M_k(S_k-2)=M_kS_{k+1}.
$$
Consequently, if $S_k>2$, so that $S_{k+1}>0$, then $\E\left[M_{k+1}\mid\cFx_k\right]=M_k$. If $S_k=2$, so that $S_{k+1}=0$, or if $S_k=S_{k+1}=0$, then $M_{k+1}=M_k$ by definition. Hence, in all cases $\E\left[M_{k+1}\mid\cFx_k\right]=M_k$, and thus $\E\left[M_{k+1}\mid\cF_k\right]=M_k$.
\end{proof}

In order to obtain a bound on the quadratic variation of (a stopped version of) $M_k$, we need to show that $S_k$ grows at least linearly in $k$ throughout the range $\nu_n$ to $(1-\eps)N$. 

\begin{lemma}\label{LC}
If\/ $\lambda_1=\lambda_2$, then, for every $\eps>0$ there exists $c>0$ such that \whp{} $S_k\ge ck$ whenever $\nu_n\le k\le (1-\eps)N$.
\end{lemma}

\begin{proof}
Assume that $\lambda_1=\lambda_2=1$. The total set of active half-edges then
evolves as in a one-type process with a single unit rate infection type. We
consider a continuous time representation of such a process, inspired by
\cite{SvanteMalwina}. As in our continuous time exploration process, each
half-edge is throughout classified as \emph{free} or \emph{paired}, and free
half-edges are labeled as \emph{active} or \emph{inactive}. All half-edges
are assigned independent unit rate exponential life lengths and, to start
the growth, two vertices are chosen uniformly at random and their half-edges
are declared active, while all other half-edges are inactive. The process
then evolves in that an active half-edge $q$ is chosen uniformly at random
and, when the life length of a free half-edge $r\neq q$ (active or inactive)
expires, then $q$ and $r$ are paired. The vertex to which $r$ is attached
becomes infected (if it was not infected already) and its remaining
half-edges are activated. This procedure is repeated until there are no
active half-edges left. It is straightforward to verify that the process is
equivalent to the two-type growth process with equal rates once types are ignored, and we furthermore ignore the time scales. Note that, in the original continuous time process, the growth is slow in
the beginning when there are few active half-edges, while in this version,
the growth is fast in the beginning when there are many free half-edges
whose life lengths compete.

We first show that a large proportion of the edges are identified in finite
time. 
\begin{claim}
For every $\eps>0$ there exists $t_0=t_0(\eps)$ such that the number of pairings up to time $t_0$ is at least $(1-\eps)N$ \whp
\end{claim}

\begin{proof}[Proof of claim]
Note that the time of the $k$th pairing is the sum of $k$ independent
exponentials with parameters $2N-1,2N-3,\ldots,2N-2k+1$. Let
$\xi_1,\xi_2,\ldots,\xi_N$ be independent and exponentially distributed with
parameter 2 and write $\xi_{(1)}<\xi_{(2)}<\cdots<\xi_{(N)}$ for the order
statistics of the $\xi_k$'s. Due to the memoryless property $\xi_{(k)}$ is
the sum of $k$ independent exponentials with parameters
$2N,2N-2,\ldots,2N-2k+2$, and it follows that the time of the $k$th pairing is
stochastically dominated by $\xi_{(k+1)}$. We are hence done if we show that
$\xi_{(\lceil(1-\eps)N\rceil+1)}\le t_0$ \whp\, for some $t_0$ or, equivalently, that the
number of $\xi_k$ that exceed $t_0$ is at most $\eps N-1$. This however
follows from the law of large numbers if we pick $t_0$ large such that
$\PP(\xi_k>t_0)<\eps$.
\end{proof}

\begin{claim}
There exists $\delta>0$ such that throughout the interval $[0,t_0]$ the proportion of uninfected vertices with degree at least $3$ is at least $\delta$ \whp
\end{claim}

\begin{proof}[Proof of claim]
Fix $d\ge3$ such that $p_d>0$. Let $V_d(t)$ denote the number of vertices of degree $d$ with all half-edges with life lengths longer than $t$. Again by the (weak) law of large numbers we have that
$$
\Bigl|\frac{1}{n}V_d(t_0)-p_d\,e^{-dt_0}\Bigr|\pto0\quad\text{as }n\to\infty.
$$
The number of uninfected vertices of degree $d$ at time $t_0$ is at least $V_d(t_0)-2$, so the claim follows.
\end{proof}

We now return to the discrete time exploration process. Recall that $\Delta S_k=S_{k+1}-S_k$ and
that $\mathcal{F}_k$ is the $\sigma$-field of events determined by the
process up to time $k$. After $k$ steps there are $2N-2k$ unpaired
half-edges and hence 
$$
\P\big(\Delta S_k=-2\bigmid\mathcal{F}_k\big)\,=\,\frac{S_k-1}{2N-2k-1}\,\le\,\frac{S_k}{2N-2k}.
$$
If the active half-edge that is paired in step $k+1$ is connected to an
inactive half-edge attached to a vertex with degree at least 3, then the
number of active half-edges increases. The degree of the vertex of the
inactive half-edge has a size biased distribution, and hence the probability that it is at
least 3 is at least as large as the proportion of uninfected vertices with
degree at least 3. Combining the above two claims we find that, for all
$k=1,2,\ldots,(1-\eps)N$, \whp 
$$
\P\big(\Delta S_k\ge1\bigmid\mathcal{F}_k\big)\,\ge\,\delta\Big(1-\frac{S_k}{2N-2k}\Big).
$$
In particular, whenever $1\le S_k\le \eps\delta N/4$, we have that
$$
\P\big(\Delta S_k=-2\bigmid\mathcal{F}_k\big)\le\delta/8\quad \text{and}\quad\P\big(\Delta S_k\ge1\bigmid\mathcal{F}_k\big)\ge\delta/2.
$$
Now, let $\zeta_1,\ldots,\zeta_N$ be i.i.d.\ random variables taking values
$-2$ and $1$ with probability $\delta/8$ and $\delta/4+\eps\delta/8$,
respectively, and otherwise the value $0$, and define $X_k:=\sum_{j=1}^k\zeta_j$. Then, by the law of
large numbers, $X_k>\eps\delta k/16$ \whp\, for all $k\ge\nu_n$, while $X_k$
is unlikely to ever exceed $\eps\delta N/4$. Moreover, since $\nu_n=o(\sqrt{n})$ by  \eqref{nun},
the number of active half-edges is unlikely to ever hit zero in the first $\nu_n$ steps.\footnote{Indeed, either $S_k$ exceeds $2\nu_n$ before reaching zero, which is good, or the probability of pairing two active half-edges is at most $2\nu_n/(N-2\nu_n)$  in each of these step, so the claim follows from the union bound.} We conclude that there is a coupling between $(S_k)_{k\ge1}$ and $(X_k)_{k\ge1}$ such that \whp
$$
S_k\ge X_k\quad\text{for all }k=1,2,\ldots,(1-\eps)N.
$$
Consequently, $S_k\ge\eps\delta k/16$ \whp\, whenever $\nu_n\le k\le(1-\eps)N$.
\end{proof}

Fix $\eps>0$ and $c$ as in \refL{LC}, and let $\tau$ be the stopping time $\min\{k\ge\nu_n:S_k< ck\}$. Thus, by \refL{LC}, \whp{} $\tau>(1-\eps)N$. Let $\tM_k:=M_{k\land \tau}$, that is, the martingale $M$ stopped at $\tau$. Then $(\tM_k)_{k=0}^N$ is also a martingale. We consider the quadratic variation of this martingale.

\begin{lemma}\label{LQ}
As \ntoo,
$$
  \E\left[\sum_{k=\nu_n}^{(1-\eps)N} |\gD \tM_k|^2\right] \to0.
$$
\end{lemma}

\begin{proof} Throughout the proof, $C$ denotes a constant, possibly depending on $\eps$ and $c$, that may be
different on each occurrence. Let $k\in[\nu_n,(1-\eps)N]$. We may suppose that $S_k\ge ck$, since otherwise $\tau\le k$ and $\gD \tM_k=0$. Then,
\begin{equation}\label{gDM}
  \gD \tM_k = \gD M_k
= \frac{\Se_k+\gD \Se_k}{S_{k}+\gD S_{k}} - \frac{\Se_k}{S_k}
= \frac{S_k\gD \Se_k-\Se_k\gD S_k}{S_k(S_{k}+\gD S_k)}.
\end{equation}
If a new vertex of degree $d$ is infected at time $k+1$, then $\gD\Se_k$ equals either $0$ or $\gD S_k=d-2$. In either case, \eqref{gDM} implies that 
$$
| \gD \tM_k|
\le
\frac{d-2}{S_{k}+d-2}
\le  \frac{d}{S_{k}+d}
\le  \frac{d}{c{k}+d}
\le  C \frac{d}{k+d}.
$$
If no new vertex is infected at time $k+1$, then $\gD S_k=-2$ and
\eqref{gDM} yields (for large $k$)
$$
| \gD\tM_k|
\le
\frac{2}{S_{k}-2}
\le  \frac{2}{c{k}-2}
\le  \frac{C}{k}.
$$
Hence, if $d\kk$ is the degree of the vertex infected at time $k+1$, with $d\kk=0$ if there is no such vertex, then
\begin{equation}\label{ele}
  \E\left[\sum_{k=\nu_n}^{(1-\eps)N} |\gD \tM_k|^2\right]
\le
C \E\left[\sum_{k=\nu_n}^{(1-\eps)N} \Bigpar{\frac{d\kk}{k+d\kk}}^2\right]
+ C\sum_{k=\nu_n}^\infty \frac{1}{k^2}.
\end{equation}

After step $k$, there are $2(N-k)$ free half-edges and hence, for each
vertex $i$ and step $k\le(1-\eps)N$, the probability that $i$ is infected in
step $k+1$, given that it has not been infected earlier, equals
$d_i/(2(N-k)-1)\leq Cd_i/n$. Hence, for any $k\le(1-\eps)N$,
\begin{equation}\label{win}
\E %\sum_{k=\nu_n}^{(1-\eps)N}
\left[\Bigpar{\frac{d\kk}{k+d\kk}}^2\right]
\le C %\sum_{k=\nu_n}^{(1-\eps)N}
\sumin \frac{d_i}{n}\Bigpar{\frac{d_i}{k+d_i}}^2
= C \frac{1}n\sumin \frac{d_i^3}{(k+d_i)^2}
=C %\sum_{k=\nu_n}^{(1-\eps)N}
\E\left[\frac{D_n^3}{(k+D_n)^2}\right].
\end{equation}
For any $d\ge1$, we have the estimates
$$
\sumk \frac{d^3}{(k+d)^2}
\le \sum_{k=1}^d \frac{d^3}{d^2} + \sum_{k=d+1}^\infty\frac{d^3}{k^2}
\le d^2 + \frac{d^3}d=2d^2
$$
and
$$
\sum_{k=\nu_n}^\infty \frac{d^3}{(k+d)^2}
\le \sum_{k=\nu_n}^\infty\frac{d^3}{(k+1)^2}
\le \frac{d^3}{\nu_n}.
$$
Hence,
\begin{equation}\label{auf}
\sum_{k=\nu_n}^\infty  \frac{D_n^3}{(k+D_n)^2}
\le {2D_n^2\land\frac{D_n^3}{\nu_n}}.
\end{equation}
By assumption, $D_n\dto D$ and $\nu_n\to\infty$, and thus $2D_n^2\land\nu_n\qw{D_n^3} \le \nu_n\qw D_n^3\pto 0$.
Furthermore, $D_n^2$ is uniformly integrable, and thus so is $2D_n^2\land\nu_n\qw{D_n^3}$. Consequently, we have by \eqref{auf} that
\begin{equation}\label{sof}
\E\left[\sum_{k=\nu_n}^\infty  \frac{D_n^3}{(k+D_n)^2}\right]
\le
\E\left[2D_n^2\land\frac{D_n^3}{\nu_n}\right]\to0.
\end{equation}
The proposition now follows from \eqref{ele}, \eqref{win} and \eqref{sof}.
\end{proof}

\begin{proof}[Proof of Proposition \ref{prop:main}]
Since $\tM_k-\tM_{\nu_n}$, with $k\ge\nu_n$, is a martingale, Doob's inequality and \refL{LQ} yield
\begin{equation*}
\E\left[\sup_{\nu_n\le k\le (1-\eps)N}\bigabs{\tM_k-\tM_{\nu_n}}^2\right]
\le 4 \E\left[\bigabs{\tM_{\floor{(1-\eps)N}}-\tM_{\nu_n}}^2\right]
=4 \E\left[\sum_{k=\nu_n}^{\floor{(1-\eps)N}-1} |\gD \tM_k|^2\right]\to0.
  \end{equation*}
Hence,  $ \sup_{\nu_n\le k\le (1-\eps)N}\bigabs{\tM_k-\tM_{\nu_n}} \pto 0$, and \eqref{lx1} follows since by \refL{LC}, \whp{} $\tau>(1-\eps)N$ and thus $M_k=\tM_k$ for $k\le(1-\eps)N$.
\end{proof}

\section{Proof of Theorem \ref{th:main}}

We can now prove Theorem \ref{th:main} by combining Proposition \ref{prop:initial} and Proposition \ref{prop:main}.

\begin{proof}
First assume that $\lambda_1=\lambda_2$. Fix $\varepsilon>0$ and let the
sequences $\nu_n$ and $t_n$ be as in Propositions \ref{prop:initial} and
\ref{prop:main}. Recall from the paragraph preceding \eqref{nun} that
$\cN_t$ denotes the number of steps (pairings of half-edges) that have been
performed at time $t$ in the continuous time exploration process. By
definition, we have that $\cM_{t_n}=M_{\cN_{t_n}}$ and, by \eqref{nun}, that
$\cN_{t_n}\geq \nu_n$ w.h.p. Hence, by Proposition \ref{prop:main},
$$
\sup_{\nu_n\le k\le (1-\eps)N}\bigabs{M_k-\cM_{t_n}} \pto 0.
$$
Furthermore, by Proposition \ref{prop:initial}, the fraction $\cM_{t_n}$ converges in distribution to a continuous random variable with support on $(0,1)$. Since a vertex that is infected in step $k+1$ in the discrete time exploration process is infected by type 1 independently with probability $M_k$, it follows from the law of large numbers that the fraction of type 1 vertices among all vertices that are infected at steps $k\in[\nu_n,(1-\eps)N]$ converges in distribution to $V$.

Recall from \eqref{nun} that $\nu_n\leq n^{1/3}$ by definition. Hence the
number of vertices that are infected before step $\nu_n$ does not exceed
$n^{1/3}$. The number of vertices that are infected after step $(1-\eps)N$
w.h.p.\ does not exceed $\eps (\E[D]+1)n$, since $N\leq (\E[D]+1)n$
w.h.p. The asymptotic fraction of vertices infected for
$k\in[\nu_n,(1-\eps)N]$ is hence at least $1-C\eps$. Since $\eps>0$ is
arbitrary, part (a) of the theorem follows.

To prove part (b), assume that $\lambda_1<\lambda_2$ and consider a modified version of the process where, after time $t_n$, the weaker type 1 infection spread with the same larger intensity $\lambda_2$ as the type 2 infection. To generate this process, we independently equip each half-edge $h$ with two independent Poisson processes $\cP^{\sss (1)}_h$ and $\cP^{\sss (2)}_h$, both with rate $\lambda_2$, and let $\check{\cP}^{\sss(1)}_h$ denote a thinned version of $\cP^{\sss (1)}_h$ where each point is kept with probability $\lambda_1/\lambda_2$, so that $\check{\cP}^{\sss (1)}_h$ is a Poisson process with rate $\lambda_1$. The process is then generated by letting the possible infection times for an active type 1 or 2 half-edge $h$ be specified by $\check{\cP}^{\sss (1)}_h$ and $\cP^{\sss (2)}_h$, respectively, up until time $t_n$, and by $\cP^{\sss (1)}_h$ and $\cP^{\sss (2)}_h$ after that time. The original process can be generated by using the thinned process $\check{\cP}^{\sss (1)}_h$ for type 1 throughout the whole time course. The corresponding discrete time processes are defined by observing the continuous time processes at the times of pairings.

Let $\check{\cS}^{\sss(i)}_t$ denote the number of active type $i$
half-edges at time $t$ in the modified process, and similarly for other
quantities. The above construction provides a coupling of the original
process and the modified process where
$\check{\cS}^{\sss(i)}_t=\cS^{\sss(i)}_t$ for $t\leq t_n$ and $i=1,2$, while
$\check{\cS}^{\sss(1)}_t\geq \cS^{\sss(1)}_t$ and
$\check{\cS}^{\sss(2)}_t\leq \cS^{\sss(2)}_t$ for $t>t_n$. It follows that
$\check{\cM}_t=\cM_t$ for $t\leq t_n$ and $\check{\cM}_t\geq \cM_t$ for
$t>t_n$.
Analogously, if $\cV^{\sss(i)}_i$ denotes the set of infected vertices of type $i$ at
time $t$, we have that $\check{\cV}^{\sss(1)}_t\supseteq \cV^{\sss(1)}_t$ and
$\check{\cV}^{\sss(2)}_t\subseteq \cV^{\sss(2)}_t$ for all $t$. Hence the number of type 1 infected vertices is at least as large in the modified process as in the original process, and it will suffice to show that the fraction of type 1 infected vertices in the
modified process converges to 0.

The modified process has equal intensities for the infection types after time $t_n$, that is, after step $\cN_{t_n}$ in the discrete time process. By \eqref{nun}, we have $\cN_{t_n}\geq \nu_n$ \whp\, and it then follows from Proposition \ref{prop:main} that
$$
\sup_{\cN_{t_n}\le k\le (1-\eps)N}\bigabs{\check{M}_k-\check{\cM}_{t_n}} \pto 0.
$$
Up to time $t_n$, on the other hand, type 1 spreads with a strictly smaller intensity and thus, by Proposition \ref{prop:initial}, the fraction $\check{\cM}_{t_n}$ converges to 0 in probability. By the same arguments as in the proof of part (a), this yields that the fraction of type 1 infected vertices in the modified process converges to 0, as desired.
\end{proof}

\section{Further work}

We have studied competing first passage percolation on the configuration model with finite variance degrees and exponential edge weights, and shown that both infection types occupy positive fractions of the vertex set if and only if they spread with the same intensity. There are several natural extensions of this work. One would be to investigate the scaling of the number of vertices of the losing type when the intensities are different. The results in \cite{regular} contain results in this direction for random regular graph and we conjecture that the results would be similar for finite variance graphs. Specifically we conjecture that, when $\lambda_1<\lambda_2$, the number of vertices occupied by type 1 is of the order $n^{\lambda_1/\lambda_2}$. In contrast to the case when the degree variance is infinite, treated in \cite{winner}, the winner hence does not take it all, but the loosing type also grows to infinity with $n$.

In \cite{regular}, also more general initial conditions are considered, where the initial number of one or both types may grow with $n$. This could also be done in our case and, in addition, one could consider initial sets where the vertices are chosen based on degree. Is it for instance possible for a weaker type to capture a positive fraction of the vertices if it can start from one or more high degree vertices, while the stronger type starts from a vertex with small degree?

Another extension would be to study more general passage time distributions, possibly different for the two types. Also in the general case, the initial growth of the types can be approximated by branching processes, but these are then not Markovian. A reasonable guess is that the possibility for both types to occupy positive fractions of the vertex set is determined by the relation between the Malthusian parameters of these branching processes, as discussed in \cite{fixspeedI}.

\end{document}